\documentclass[12pt]{amsart}
\usepackage{pgf,tikz,pgfplots} 
\pgfplotsset{compat=1.14}
\usepackage{mathrsfs}
\usetikzlibrary{arrows}
\usetikzlibrary{patterns}
\usepackage{pgfplots}
\usetikzlibrary{intersections, pgfplots.fillbetween}
\usepackage[colorlinks=true,urlcolor=blue, citecolor=red,linkcolor=blue,linktocpage,pdfpagelabels, bookmarksnumbered,bookmarksopen]{hyperref}
\usepackage[hyperpageref]{backref}
\usepackage{cleveref}
\usepackage{xcolor}
\usepackage{amsthm} 
\usepackage{latexsym,amsmath,amssymb}
\usepackage{accents}
   \usepackage[colorinlistoftodos,prependcaption,textsize=tiny]{todonotes}
\usepackage{a4wide}
\usepackage{soul}
\usepackage{mathtools} 
\usepackage{xparse} 
  
\pgfdeclarelayer{ft}
\pgfdeclarelayer{bg}
\pgfsetlayers{bg,main,ft}

\title{Existence of nontrival $n$-harmonic maps via min-max methods}

\author{Dorian Martino}

\address[Dorian Martino]{
Department of Mathematics, ETH Zurich, Rämistrasse 101, 8092 Zurich, Switzerland}
\email{dorian.martino@math.ethz.ch}
\author{Katarzyna Mazowiecka}
\address[Katarzyna Mazowiecka]{
Institute of Mathematics,%
University of Warsaw,
Banacha 2,
02-097 Warszawa, Poland}
\email{k.mazowiecka@mimuw.edu.pl}
\author{Armin Schikorra}
\address[Armin Schikorra]{Department of Mathematics,
University of Pittsburgh,
301 Thackeray Hall,
Pittsburgh, PA 15260, USA}
\email{armin@pitt.edu}
\definecolor{indigo}{rgb}{0.29, 0.0, 0.51}
\definecolor{p1}{gray}{0.4}
\definecolor{p2}{gray}{0.6}
\definecolor{p3}{gray}{0.98}
\definecolor{p4}{gray}{0.8}
\definecolor{p5}{gray}{0.9}

plus 6pt minus 12pt
\def\eps{\varepsilon}

\def\B{{B}}

\def\n{{\mathcal M}}
\newcommand{\dif}{\,\mathrm{d}}

\def\N{{\mathbb N}}
\def\n{{\mathcal{M}}}
\def\n{{\mathcal N}}

\def\S{{\mathbb S}}

\renewcommand{\div}{{\rm div}}

\newtheorem{theorem}{Theorem}
\newtheorem{lemma}[theorem]{Lemma}
\newtheorem{corollary}[theorem]{Corollary}

\newcommand{\loc}{\mathrm{loc}}

\def\dist{{\rm dist\,}}

\def\supp{{\rm supp\,}}
\def\loc{{\rm loc}}

\newcommand{\dx}{\dif x}

\newcommand{\R}{\mathbb{R}}

\newcommand{\brac}[1]{\left (#1 \right )}
\newcommand{\abs}[1]{\left |#1 \right |}

\newcommand{\barint}{
\rule[.036in]{.12in}{.009in}\kern-.16in \displaystyle\int }

\newcommand{\barcal}{\mbox{$ \rule[.036in]{.11in}{.007in}\kern-.128in\int $}}

\def\mvint_#1{\mathchoice
          {\mathop{\vrule width 6pt height 3 pt depth -2.5pt
                  \kern -8pt \intop}\nolimits_{\kern -3pt #1}}%
          {\mathop{\vrule width 5pt height 3 pt depth -2.6pt
                  \kern -6pt \intop}\nolimits_{#1}}%
          {\mathop{\vrule width 5pt height 3 pt depth -2.6pt
                  \kern -6pt \intop}\nolimits_{#1}}%
          {\mathop{\vrule width 5pt height 3 pt depth -2.6pt
                  \kern -6pt \intop}\nolimits_{#1}}}

\numberwithin{theorem}{section} \numberwithin{equation}{section}

\newcommand{\aleq}{\precsim}
\newcommand{\ageq}{\succsim}

\newcommand{\mfdN}{\mathcal{N}}

\let\latexchi\chi
\makeatletter
\renewcommand\chi{\@ifnextchar_\sub@chi\latexchi}
\newcommand{\sub@chi}[2]{
  \@ifnextchar^{\subsup@chi{#2}}{\latexchi^{}_{#2}}%
}
\newcommand{\subsup@chi}[3]{
  \latexchi_{#1}^{#3}%
}
\makeatother

\makeatletter
\def\tikz@arc@opt[#1]{
  {%
    \tikzset{every arc/.try,#1}%
    \pgfkeysgetvalue{/tikz/start angle}\tikz@s
    \pgfkeysgetvalue{/tikz/end angle}\tikz@e
    \pgfkeysgetvalue{/tikz/delta angle}\tikz@d
    \ifx\tikz@s\pgfutil@empty%
      \pgfmathsetmacro\tikz@s{\tikz@e-\tikz@d}
    \else
      \ifx\tikz@e\pgfutil@empty%
        \pgfmathsetmacro\tikz@e{\tikz@s+\tikz@d}
      \fi%
    \fi
    \tikz@arc@moveto
    \xdef\pgf@marshal{\noexpand%
    \tikz@do@arc{\tikz@s}{\tikz@e}
      {\pgfkeysvalueof{/tikz/x radius}}
      {\pgfkeysvalueof{/tikz/y radius}}}%
  }%
  \pgf@marshal%
  \tikz@arcfinal%
}
\let\tikz@arc@moveto\relax
\def\tikz@arc@movetolineto#1{%
  \def\tikz@arc@moveto{\tikz@@@parse@polar{\tikz@arc@@movetolineto#1}(\tikz@s:\pgfkeysvalueof{/tikz/x radius} and \pgfkeysvalueof{/tikz/y radius})}}
\def\tikz@arc@@movetolineto#1#2{#1{\pgfpointadd{#2}{\tikz@last@position@saved}}}
\tikzset{%
  move to start/.code=\tikz@arc@movetolineto\pgfpathmoveto,%
  line to start/.code=\tikz@arc@movetolineto\pgfpathlineto}
\makeatother

\begin{document}

\begin{abstract}
For any $n \geq 3$ and any closed manifold $\n$ with $\pi_{n+k}(\n) \neq \{0\}$ for some $k \geq 0$, we establish the existence of nontrivial $n$-harmonic maps from $\S^n$ into $\n$. When $k\geq 1$, these maps naturally appear as bubbling limits of $p$-harmonic maps with $p > n$, obtained by min-max constructions in the limit $p \to n^+$.
\end{abstract}

\sloppy

\maketitle
\sloppy

\section{Introduction}
In this note, we study the existence of nontrivial $n$-harmonic maps from the sphere $\S^n$ into a smooth, compact manifold $\n \subset \R^N$ without boundary. These maps are critical points of the $n$-energy
\[
 \mathcal{E}_n(u) \coloneqq \int_{\S^n} |\nabla u|^n,\quad u\colon \S^n \to \n,
\]
i.e., distributional solutions to
\begin{equation}
 -\div(|\nabla u|^{n-2} \nabla u) = |\nabla u|^{n-2} A(u)[\nabla u,\nabla u].
\end{equation}
Here $A$ denotes the second fundamental form. Our main result is the following.
\begin{theorem}\label{th:main}
Let $n \geq 3$. If $\pi_{n+k}(\n) \neq \{0\}$ for some $k \geq 0$, then there exists a nontrivial, $C^{1,\alpha}$-regular $n$-harmonic map from $\S^n$ into $\n$, where $\alpha$ is a small positive number.
\end{theorem}
Throughout the paper we assume $n\ge3$. The two-dimensional case can be treated by the same arguments with minor modifications; since the corresponding results for harmonic maps are already well known (see \cite{Lamm06} and references therein), we do not include them here.

The existence result in this theorem is in line with general expectations in the theory of harmonic maps. In the case $k=0$, that is, when $\pi_{n}(\n) \neq \{0\}$, it follows from the classical Sacks--Uhlenbeck theory \cite{SacksU81}, since $\pi_{n}(\n)$ is generated by homotopy classes for which the $n$-energy admits a minimizer. The corresponding result for the $n$-harmonic maps was proved by \cite{Nakauchi-Takakuwa,Kawai-Nakauchi-Takeuchi}. See also \cite{Mazowiecka-Schikorra23} for related minimization results in a fractional framework.

When $\pi_{n}(\n)= \{0\}$ but $\pi_{n+k}(\n) \ni \alpha \neq 0$, one is naturally lead to seek critical points of the $n$-energy via a $\min$-$\max$ construction of the form
\[
 \inf_{\Phi \in \alpha}\ \max_{t\in [-1,1]^k}\ \mathcal{E}_n\big( \Phi(\cdot,t) \big).
\]
This is the equivalent approach to Birkhoff's min-max construction of nontrivial closed geodesics on manifolds diffeomorphic to the sphere $\S^2$. A direct implementation of this strategy at the critical exponent $n$ is, however, obstructed by the failure of the Palais--Smale condition in $W^{1,n}$.

In the recent work \cite{DML25}, this difficulty is overcome by developing a refined analytical framework, leading to a proof of the theorem \Cref{th:main} in dimension $n=3$, and in dimensions $n\geq 4$ in the case of manifolds which are spheres or Lie groups. The present work provides an alternative approach, which applies to all dimensions $n\ge3$ and to arbitrary smooth closed target manifolds $\n$.

The starting point is classical. For $p > n$, we consider the perturbed energy
\[
 \mathcal{E}_p(u) \coloneqq \int_{\S^n} |\nabla u|^p, \quad  u\colon \S^n \to \n.
\]
For $p>n$, the functional satisfies the Palais--Smale condition, allowing one to construct non-trivial critical points of the associated Euler--Lagrange equation
\begin{equation}\label{eq:pharmonic}
 -\div(|\nabla u|^{p-2} \nabla u) = |\nabla u|^{p-2} A(u)[\nabla u,\nabla u].
\end{equation}
These are obtained through the min-max scheme
\[
 \beta_p \coloneqq \min_{\Phi \in \alpha}\ \max_{ t\in [-1,1]^k}\ \mathcal{E}_p \big(\Phi(\cdot,t) \big).
\]
Provided that $\liminf_{p \to n} \beta_p > 0$, \Cref{th:main} follows  from a one-bubble compactness result: for any sequence $p \to n^+$, $p>n$ and any corresponding sequence of $p$-harmonic maps $u_p \in W^{1,p}(\S^n,\n)$, one of the following holds, up to a subsequence. Either $u_p$ converges strongly in $W^{1,n+\sigma}$  for some $\sigma > 0$, meaning that the limit is a nontrivial $n$-harmonic map; or there exists at least one nontrivial bubble, which, after rescaling, is a nontrivial $n$-harmonic map.

{\bf Outline.} Since the case $k=0$ is known (see for instance \cite{W1986,W1988} or \cite[Theorem 0]{Kawai-Nakauchi-Takeuchi}), we focus on the case $k \geq 1$. That is, throughout the remainder of this article we assume $\pi_{n+k}(\mfdN)\neq \{0\}$. In \Cref{s:prelim} we establish the validity of the min-max construction for $p$-harmonic maps and $p > n$. \Cref{s:epsreg} explains the main analytical tool of the paper. There we prove an $\eps$-regularity result for $p$-harmonic maps under a uniform $p$-energy bound. Instead of deriving higher \emph{second-order} regularity, we obtain a fractional gain in the differentiability, showing that solutions belong locally to $W^{1+\delta, p}$ for some $\delta>0$. This fractional improvement is sufficient to ensure the compactness needed in the bubbling analysis. The argument relies on the contemporary regularity theory for the $p$-Laplacian combined with the classical stability estimate in the spirit of Iwaniec--Sbordone \cite{Iwaniec-Sbordone}, which provides a uniform control of the $(n+\sigma)$-energy for some $\sigma>0$.

Finally, in \Cref{s:main} we combine the previous results to prove the main theorem. Using the compactness provided by the fractional regularity improvement and a one-bubble extraction argument, we show the existence of a nontrivial $n$-harmonic map.

{\bf Acknowledgment.} We thank Simon Nowak for pointing out \cite{BDW20} and helpful discussion.

The project is co-financed by: (DM) Swiss National Science Foundation, project SNF 200020\textunderscore219429; (KM) the Polish National Agency for Academic Exchange within Polish Returns Programme - BPN/PPO/2021/1/00019/U/00001;
(KM) the National Science Centre, Poland grant No. 2023/51/D/ST1/02907;
 (AS) NSF Career DMS-2044898;
(AS) Alexander von Humboldt Research Fellow (2023-2025).

 {\bf Notation.} Throughout the paper, $B(x,r)$ denotes the (open) ball of radius $r>0$ centred at $x$. When the ambient space is $\R^n$, this is the Euclidean ball; when the ambient space is $\S^n$, it is the geodesic ball. The meaning will always be clear from the context.

\section{Preliminaries: Existence of \texorpdfstring{$p$}{p}-harmonic maps}\label{s:prelim}
Let $\alpha \in \pi_{n+k}(\mfdN)$ be nontrivial. We identify $\S^{n+k} \simeq \S^n \times [-1,1]^k$, with the usual identification on $\partial[-1,1]^k$, and consider for $p > n$
\begin{equation}\label{eq:betapvvalue}
 \inf_{\Phi \in \alpha}\ \max_{t \in [-1,1]^k}\ \mathcal{E}_p(\Phi(\cdot,t))=\inf_{\Phi \in \alpha}\ \max_{t \in [-1,1]^k}\ \int_{\S^n} |\nabla \Phi(x,t)|^p  \dx\eqqcolon \beta_p.
\end{equation}
We need the following.
\begin{lemma}\label{la:minimalbeta}
Let $\beta_p$ be the value defined in \eqref{eq:betapvvalue}. We have
\[\liminf_{p \to n^+} \beta_p > 0. \]
\end{lemma}
\begin{proof}
By H\"older inequality it suffices to show
\[
 |\S^n|^{1-\frac{p}{n}} (\beta_p)^{\frac{n}{p}} \geq \inf_{\Phi \in \alpha} \max_{t \in [-1,1]^k} \int_{\S^n} |\nabla \Phi(x,t)|^n \dx > 0.
\]
Assume on the contrary that for a suitably small number $\eps > 0$ (chosen below) there exists a smooth $\Phi \colon \S^{n}\times[-1,1]^k \to \n$ for which we have
\begin{equation}\label{eq:smallness-contradiction}
 \max_{t \in [-1,1]^k} \int_{\S^n} |\nabla \Phi(x,t)|^n \dx < \eps.
\end{equation}
For each $t \in [-1,1]^k$, let $\Phi^h(\cdot,t)$ denote the harmonic extension of $\Phi(\cdot,t)\colon \S^n \to \n$  to a map $\Phi^h(\cdot,t)\colon \B^{n+1} \to \R^N$.
If $P(x,\xi) = \frac{1-|x|^2}{\omega_n |x-\xi|^{n+1}}$ is the Poisson kernel on the ball $B^{n+1}$, then we have 
\[
 \Phi^h(x,t) = \int_{\S^n} P(x,\xi) \Phi(\xi,t)\dif \xi.
\]
Since we have $\int_{\S^n}P(x,\xi)\dif \xi = 1$, the following holds for any $\zeta \in \S^n$
\[
\begin{split}
 \dist(\Phi^h(x,t),\n) &\le \abs{\Phi^h(x,t) - \Phi(\zeta,t)} \le \int_{\S^n} P(x,\xi) \abs{\Phi(\xi,t) - \Phi(\zeta,t)}\dif \xi.
\end{split}
 \]
Multiplying both sides by $P(x,\zeta)$ and integrating over $\S^n$ with respect to $\zeta$ we obtain
\begin{equation}\label{eq:distanceestimate}
 \begin{split}
  \dist(\Phi^h(x,t),\n) &\le \int_{\S^n}\int_{\S^n} P(x,\xi) P(x,\zeta) \abs{\Phi(\xi,t) - \Phi(\zeta,t)}\dif \xi \dif \zeta. 
 \end{split}
\end{equation}
If $|x|\le \frac12$, then we have $|x-\xi|\ge \frac12$ for any $\xi\in \S^n$ and thus, it holds
\[
 \int_{\S^n}|P(x,\xi)|^{n'}\dif \xi \aleq \int_{\S^n}\brac{\frac{1}{|x-\xi|^{n+1}}}^{n'} \dif \xi \aleq 1.
\]
Hence, for $|x|\le \frac12$, we have, using H\"{o}lder and Poincar\'{e} inequality
\[
\begin{split}
 \dist(\Phi^h(x,t),\n) 
 &\le \|P(x,\cdot)\|_{L^{n'}(\S^n)}^2 \brac{\int_{\S^n}\int_{\S^n} \abs{\Phi(\xi,t) - \Phi(\zeta,t)}^n\dif \xi \dif \zeta}^{\frac 1n}\\
 &\aleq \brac{\int_{\S^n}|\nabla \Phi(x,t)|^n \dif \xi}^{\frac 1n} \aleq \eps^{\frac 1n}.
\end{split}
 \]
If $|x|\ge \frac12$, then we let $r=1-|x|$ and $X=\frac{x}{|x|}$ and we write 
\[
A(X,r,0) \coloneqq B(X,r), \quad  A(X,r,k)\coloneqq B(X,2^{k}r)\setminus B(X,2^{k-1}r) \text{ for }k\ge 1.
\]
Let us denote by $\theta = \dist_{\S^n}(\frac{x}{|x|},\xi)$. Then, we have
\[
 |x-\xi|^2 = |x|^2+1-2|x|\frac{x}{|x|}\cdot \xi = |x|^2+1-2|x|\cos\theta = (1-|x|)^2+2|x|(1-\cos\theta).
 \]
For every $\theta\in[0,1]$ we have the elementary inequality $\frac{2}{\pi^2}\theta^2\leq 1-\cos \theta$, thus
\[
 |x-\xi|^2\ageq (1-|x|)^2+|x|\theta^2 \ageq (1-|x|)^2 + \theta^2.
\]
Hence, for $\xi \in A(X,r,k)$ we have $\theta \approx 2^k r$ and $|x-\xi|\ageq 2^k r$. Thus, from \eqref{eq:distanceestimate} we get for $|x|\ge \frac 12$
\[
\begin{split}
 \dist(\Phi^h(x,t),\n) 
 &\aleq \sum_{k\in \N\cup\{0\}}\sum_{\ell\in\N\cup\{0\}}\int_{A(X,k,r)}\int_{A(X,\ell,r)} \abs{\Phi(\xi,t) - \Phi(\zeta,t)}\frac{r}{(2^k r)^{n+1}}\frac{r}{(2^\ell r)^{n+1}}\dif \xi \dif \zeta\\
 &\aleq \sum_{k\in \N\cup\{0\}}\sum_{\ell\in\N\cup\{0\}}2^{-k} 2^{-\ell}\barint_{A(X,k,r)}\barint_{A(X,\ell,r)} \abs{\Phi(\xi,t) - \Phi(\zeta,t)}\dif \xi \dif \zeta\\
 &\aleq \sum_{k\in \N\cup\{0\}}\sum_{\ell\in\N\cup\{0\}}2^{-k} 2^{-\ell} \brac{\barint_{A(X,k,r)}\barint_{A(X,\ell,r)} \abs{\Phi(\xi,t) - \Phi(\zeta,t)}^n\dif \xi \dif \zeta}^\frac{1}{n}\\
 &\aleq \sum_{k\in \N\cup\{0\}}\sum_{\ell\in\N\cup\{0\}}2^{-k} 2^{-\ell} \|\nabla \Phi(\cdot,t)\|_{L^n(\S^n)} \aleq \eps^{\frac 1n}.
 \end{split}
\]
Thus, $\left\| \dist(\Phi^h,\n) \right\|_{L^{\infty}(B^{n+1}\times[-1,1]^k)} \aleq \eps^{1/n}$.

Let $\pi_{\n}\colon B_{\delta}(\n) \to \n$ denote the nearest point projection. We then define
\[
 \Psi\coloneqq \pi_{\n} \circ \Phi^h \colon \B^{n+1} \times [-1,1]^k \to \n.
\]
By construction, $\Psi$ is continuous and satisfies
$
\Psi  \rvert_{\S^n\times [-1,1]^k} = \Phi.
$
Hence, we have $\left[\Psi \rvert_{\S^n\times [-1,1]^k} \right] = [\Phi]$ in $\pi_{k+n}(\n)$. On the other hand, the map $\Psi\rvert_{\S^n\times [-1,1]^k}$ is contractible, therefore $\left[\Psi\rvert_{\S^n\times [-1,1]^k} \right]=[\Phi]=0$. This is a contradiction with the assumption that $\alpha=[\Phi]$ is nontrivial.
\end{proof}

Since $p > n$ the functional $\mathcal{E}_p$ satisfies the Palais--Smale condition. In view of \Cref{la:minimalbeta}, we may apply the mountain-pass lemma (see, e.g., \cite[Theorem 6.1]{StruweVariationalMethods}) which yields the existence of a nontrivial critical point  realizing the min-max value $\beta_p$.
Since $p>n$, the resulting $p$-harmonic map $u_p$ is \emph{H\"{o}lder continuous} by the Morrey--Sobolev embedding. Moreover, following the arguments of \cite[Section 3]{HLp} (see also \cite[Theorem 1.3]{M2025}) one concludes that $u \in C^{1,\alpha}$ for some $\alpha > 0$.
We have thus established the following result.
\begin{theorem}\label{th:minmaxpharmmaps}
If $\pi_{n+k}(\n) \neq \{0\}$ for some $k \geq 1$, then for any $p > n$ there exists a $p$-harmonic map $u_p\in W^{1,p}\cap C^{1,\alpha}(\S^n,\n)$, i.e., a solution to \eqref{eq:pharmonic}, and we have
\[
 \int_{\S^n} |\nabla u_p|^p =\beta_p=\min_{\Phi \in \alpha}\ \max_{t \in [-1,1]^k}\ \mathcal{E}_p\big(\Phi(\cdot,t)\big).
\]
\end{theorem}
As a corollary, using a standard covering argument, see, e.g., \cite[Proposition 4.3 and Theorem 4.4]{SacksU81},  we have the following disjunction of cases for the limiting behaviour of $u_p$ as $p\to n^+$.
\begin{corollary}\label{co:minmaxpharmmaps}

Fix $\eps > 0$. If $\pi_{n+k}(\n) \neq \{0\}$ for some $k \geq 1$, then there exists a sequence  $p_j \xrightarrow{j \to\infty} n^+$ and a sequence of $p_j$-harmonic map $u\in W^{1,p_j}\cap C^{1,\alpha_j}(\S^n,\n)$, i.e., solutions to \eqref{eq:pharmonic} with
\[
 \liminf_{j \to \infty} \int_{\S^n} |\nabla u_{p_j}|^{p_j} > 0,
\]
such that one of the following holds.
\begin{itemize} 
 \item (No bubble) For all $x \in \S^{n}$ there exists $\rho_x > 0$ such that 
 \[
  \limsup_{j \to \infty} \int_{B(x,\rho_x)} |\nabla u_{p_j}|^{n} < \eps^n.
 \]
\item (At least one isolated bubble) There exists at least one $\bar{x} \in \S^n$ and a radius $R > 0$ with the following property
 \[
  \forall \rho > 0,\qquad \liminf_{j \to \infty} \int_{B(\bar{x},\rho)} |\nabla u_{p_j}|^{n} \geq \eps^n, 
 \]
and for any $x \in \S^n$, $0<\dist_{\S^n}(x,\bar{x})<R$ there exists some $\rho_x > 0$ such that 
\[
  \limsup_{j \to \infty} \int_{B(x,\rho_x)} |\nabla u_{p_j}|^{n} \leq \eps^n.
\]
\end{itemize}
\end{corollary}

\section{An \texorpdfstring{$\eps$}{eps}-regularity Theorem}\label{s:epsreg}
The main technical ingredient of the paper is an $\eps$-regularity theorem: it says that there is a threshold $\eps > 0$ such that for any (geodesic) ball $B(x,\rho)$ where the $n$-energy drops below $\eps$ we have strong convergence (up to subsequence).

\begin{theorem}\label{th:epsreg}
There exists an $\eps > 0$ such that for each $p_0>n$ there exist $p_1 > n$ and $t_0 > 0$ with the property that for all $p \in [n,p_1]$ and any $\Lambda > 0$ the following holds.

Assume $u \in W^{1,p_0}(B(x_0,r),\n)$ is a $p$-harmonic map which satisfies
\[
 \|\nabla u\|_{L^n(B(x_0,r))} < \eps \quad \text{and} \quad
\|\nabla u\|_{L^p(B(x_0,r))} \leq \Lambda.
\]
Then it holds
\[
 [\nabla u]_{W^{t_0,p}(B(x_0,r/2))} \leq C(n,\Lambda,r).
\]
\end{theorem}

We begin with the following standard consequence of Iwaniec--Sbordone stability result \cite[Theorem 4]{Iwaniec-Sbordone} for the $p$-Laplace equation.
\begin{lemma}[Iwaniec--Sbordone a priori estimate]\label{la:iwaniec}
Let $n\geq 3$, fix a compact manifold $\n\subset \R^N$ and let $A$ be its second fundamental form. There exists $\eps > 0$, such that for any $p_0 > n$ there exists $p_1\in(n,p_0)$ and $\sigma > 0$ so that the following holds whenever $p \in [n,p_1]$.

Assume that $u \in W^{1,p_0}(B(x_0,10r),\n)$ satisfies
\begin{equation}\label{eq:pharm}
 \div(|\nabla u|^{p-2} \nabla u) = |\nabla u|^{p-2} A(u)[\nabla u,\nabla u] \quad \text{in $\B(x_0,10r)$},
\end{equation} 
and 
\begin{equation}\label{eq:smallnessassumption}
 \|\nabla u\|_{L^n(B(x_0,10r))} < \eps.
\end{equation}
Then it holds
\[
 \|\nabla u\|_{L^{n+\sigma}(B(x_0,r))} \leq C(\n,n)\brac{1+\|\nabla u\|_{L^{p}(B(x_0,10r))}^{p-1}}^{\frac{n+1}{n-1}} .
\]
\end{lemma}

\begin{proof}
For simplicity of presentation, we assume that $B(x_0,10r)$ is a Euclidean ball $B(0,10)$\footnote{In the case of a geodesic ball in $\S^n$, one works in normal coordinates; the $n$-Laplacian then has smooth coefficients, and all arguments and estimates carry over unchanged.}.

First we localize. Let $\tilde{\eta} \in C_c^\infty(B(0,2))$ such that $\eta \equiv 1$ in $B(0,1)$ and set 
\[
 \tilde{u} \coloneqq \eta (u-(u)_{B(0,1)}).
\]
We multiply \eqref{eq:pharm} with $\eta^{p-1}$ and obtain a system valid on $\R^n$
\[
 \div(|\nabla u|^{p-2} \nabla u)\eta^{p-1} = |\nabla u|^{p-2} A(u)[\nabla u,\nabla u]\eta^{p-1} \quad \text{in $\R^n$}.
\]
Thus, it holds
\begin{equation}\label{eq:cutoff_syst}
\begin{split}
 \div(\eta^{p-1} |\nabla u|^{p-2} \nabla u) 
 &= |\nabla u|^{p-2} A(u)[\nabla u,\nabla u]\eta^{p-1} + (p-1) \eta^{p-2} \nabla \eta |\nabla u|^{p-2} \nabla u\\
 & = |\nabla \tilde{u}|^{p-2} A(u)[\nabla \tilde{u},\nabla u] + g,\\
 \end{split}
\end{equation} 
where the map $g$ verifies the following estimate
\begin{equation}\label{eq:estimateong}
 \|g\|_{L^{p'}(\R^n)} \aleq \|\nabla u\|_{L^p(B(0,2))}^{p-1},
\end{equation}
where the constant depends on $\|u\|_{L^\infty}$, that is on $\mathcal N$.

We define the following quantity
\[
 F\coloneqq |\nabla \tilde{u}|^{p-2} \nabla \tilde{u}- \eta^{p-1} |\nabla u|^{p-2} \nabla u .
\]
Since $p > 2$, the map $F$ satisfies the following estimate
\begin{equation}\label{eq:estimateonF}
 \|F\|_{L^{\frac{p}{p-2}}(\R^n)} \aleq \|\nabla u\|_{L^p(B(0,2))}^{p-2},
\end{equation}
where the constant depends on $\|u\|_{L^\infty}$, that is on $\mathcal N$.

Plugging $F$ into \eqref{eq:cutoff_syst}, we arrive at
\begin{equation} \label{eq:pharm2}
\begin{split}
 \div(|\nabla \tilde{u}|^{p-2} \nabla \tilde{u})& = |\nabla \tilde{u}|^{p-2} A(u)[\nabla \tilde{u},\nabla u] + g + \div F.
 \end{split}
\end{equation} 
Since by assumption $\tilde{u} \in W^{1,p_0}(B(0,3))$, then for some $0<\sigma \le p_0-p$ with $\sigma<1$, we can find (see \cite[Theorem 8.1]{Iwaniec} and \cite[Theorem 4]{Iwaniec-Sbordone}) via Hodge decomposition $\varphi \in W^{1,\frac{p+\sigma}{1+\sigma}}_0(B(0,3))$ and $B \in L^{\frac{p+\sigma}{1+\sigma}}(B(0,3))$ such that
\[
 |\nabla \tilde{u}|^{\sigma} \nabla \tilde{u} = \nabla \varphi + B \quad \text{in $B(0,3)$},
\]
with the estimates
\begin{equation} \label{eq:varphi}
 \|\nabla \varphi\|_{L^{\frac{p+\sigma}{1+\sigma}}(B(0,3))} \aleq \|\nabla \tilde{u}\|_{L^{p+\sigma}(B(0,3))}^{1+\sigma}
\end{equation} 
and 
\begin{equation} \label{eq:B}
 \|B\|_{L^{\frac{p+\sigma}{1+\sigma}}(B(0,3))} \aleq \sigma \|\nabla \tilde{u}\|_{L^{p+\sigma}(B(0,3))}^{1+\sigma}.
\end{equation}
The constants in inequalities \eqref{eq:varphi} and \eqref{eq:B} are independent of $p$ and $\sigma$, provided $p\in [n,n+1]$ and $\sigma\in[0,1]$. 

Since $\supp \tilde{u} \subset B(0,3)$, we have
\[
\begin{split}
 &\|\nabla \tilde{u} \|_{L^{p+\sigma}(\R^n)}^{p+\sigma}\\
 &=\int_{B(0,3)} |\nabla \tilde{u}|^{p-2} \nabla \tilde{u} \cdot \nabla \varphi + \int_{B(0,3)} |\nabla \tilde{u}|^{p-2} \nabla \tilde{u} \cdot B\\
 &\leq\abs{\int_{B(0,3)} |\nabla \tilde{u}|^{p-2} \nabla \tilde{u} \cdot \nabla \varphi} + \|\nabla \tilde{u} \|_{L^{p+\sigma}(B(0,3))}^{p-1} \|B\|_{L^{\frac{p+\sigma}{1+\sigma}}(B(0,3))}\\
 &\le C\abs{\int_{B(0,3)} |\nabla \tilde{u}|^{p-2} \nabla \tilde{u} \cdot \nabla \varphi} + C\sigma \|\nabla \tilde{u} \|_{L^{p+\sigma}(\R^n)}^{p+\sigma}.
 \end{split}
\]
The  constant $C$ is independent of $\sigma$. Thus for $\sigma = (2C)^{-1}$, we can absorb the second term on the right-hand side, and arrive at 
\begin{equation}\label{eq:ourestimate}
 \|\nabla \tilde{u} \|_{L^{p+\sigma}(\R^n)}^{p+\sigma} \aleq\abs{\int_{B(0,3)} |\nabla \tilde{u}|^{p-2} \nabla \tilde{u} \cdot \nabla \varphi},
\end{equation}
with constants depending on $n$, but uniform in $p \in [n,n+1]$. 
For the fixed $\sigma = (2C)^{-1}$ that is not going to change anymore, let $p_1 \in (n,n+1)$ be so that 
\[
\frac{p_1+\sigma}{1+\sigma} < n. 
\]
This is possible because $n \geq 2$. Then we have
\begin{equation} \label{eq:sup_p}
 \sup_{p \in [n,p_1]} \frac{p+\sigma}{1+\sigma} < n .
\end{equation} 
We continue with \eqref{eq:ourestimate} (where constants are independent of $p \in [n,p_1]$), using \eqref{eq:pharm2}, we obtain
\begin{equation}\label{eq:number1}
\begin{split} 
	\|\nabla \tilde{u} \|_{L^{p+\sigma}(\R^n)}^{p+\sigma}
	\aleq&\abs{\int_{B(0,3)} |\nabla \tilde{u}|^{p-2} \nabla \tilde{u} \cdot \nabla \varphi}\\
	\aleq&\int_{B(0,3)} |\nabla u| |\nabla \tilde{u}|^{p-1} \abs{\varphi}+\int_{B(0,3)} \abs{g} \abs{\varphi}
	+\int_{B(0,3)} \abs{F} \abs{\nabla \varphi}.
\end{split} 
\end{equation}
We apply Hölder inequality to the first term with exponents 
\[
\frac{1}{n} + \frac{p-1}{p+\sigma} + \frac{1}{q} = 1,
\]
where
\[
	\frac{1}{q} = 1-\frac{p-1}{p+\sigma} - \frac{1}{n} = \frac{1+\sigma}{p+\sigma} - \frac{1}{n}\leq 1.
\]
Hence, we obtain that $q = \left( \frac{p+\sigma}{1+\sigma}\right)^*$.
By the Sobolev--Poincar\'{e} inequality (which we can apply with constant independent of $p$ because \eqref{eq:sup_p} holds) and by \eqref{eq:varphi} we get
\begin{equation}\label{eq:number2}
\begin{split} 
	\int_{B(0,3)} |\nabla u| |\nabla \tilde{u}|^{p-1} \abs{\varphi} & \leq \|\nabla u \|_{L^n(B(0,3))}\, \|\nabla \tilde{u} \|_{L^{p+\sigma}(\R^n)}^{p-1} \, \|\varphi\|_{L^{\left( \frac{p+\sigma}{1+\sigma}\right)^* }(B(0,3))} \\
	 & \aleq \|\nabla u \|_{L^n(B(0,3))}\, \|\nabla \tilde{u} \|_{L^{p+\sigma}(\R^n)}^{p-1} \, \|\nabla \varphi\|_{L^{ \frac{p+\sigma}{1+\sigma} }(B(0,3))} \\
	 & \aleq \|\nabla u \|_{L^n(B(0,3))}\, \|\nabla \tilde{u} \|_{L^{p+\sigma}(\R^n)}^{p+\sigma}.
\end{split} 
\end{equation}
We also have $\left(\frac{p+\sigma}{1+\sigma}\right)^*>p$. Indeed,
the inequality is equivalent to $\frac{p+\sigma}{1+\sigma}>\frac{np}{n+p}$, which in turn is equivalent to 
\[
p^2+\sigma(n-(n-1)p)>0.
\]
Since $p\ge n\ge 2$, we have $n-(n-1)p\le 0$, and using $0<\sigma<1$ we obtain
\[
p^2+\sigma(n-(n-1)p)
\ge p^2+(n-(n-1)p)
= p(p-n)+n+p
>0.
\]
This proves the claimed inequality $\left(\frac{p+\sigma}{1+\sigma}\right)^*>p$.

Hence, we obtain using \eqref{eq:varphi} and then \eqref{eq:estimateong}
\begin{equation}\label{eq:number3}
\begin{split}
	\int_{B(0,3)} \abs{g} \abs{\varphi} 
	&\leq \|g\|_{L^{p'}(B(0,3))}\, \|\varphi \|_{L^p(B(0,3))} \\
	&\aleq \|g\|_{L^{p'}(B(0,3))}\, \|\nabla\varphi\|_{L^{\frac{p+\sigma}{1+\sigma}}(B(0,3))}\\
	&\aleq \|\nabla u\|_{L^p(B(0,2))}^{p-1}\|\nabla\tilde u\|_{L^{p+\sigma}(B(0,3))}^{1+\sigma}.
\end{split}
\end{equation}
To estimate the last term of \eqref{eq:number1}, we note that since $\sigma<1$ we have $\frac{p}{2} \le \frac{p+\sigma}{1+\sigma}$. Hence, from H\"{o}lder inequality, \eqref{eq:varphi}, and \eqref{eq:estimateonF} we obtain
\begin{equation}\label{eq:inequalityonF}
 \begin{split}
  \int_{B(0,3)} \abs{F} \abs{\nabla \varphi} 
  &\le \|F\|_{L^{\frac{p}{p-2}}(B(0,3))}\|\nabla \varphi\|_{L^{\frac{p}{2}}(B(0,3))} \aleq \|\nabla u\|_{L^p(B(0,2))}^{p-2} \|\nabla \tilde{u}\|_{L^{p+\sigma}(B(0,3))}^{1+\sigma}.
 \end{split}
\end{equation}
Combining \eqref{eq:number1} with \eqref{eq:number2}, \eqref{eq:number3}, and \eqref{eq:inequalityonF} we obtain
\[
\begin{split}
 \|\nabla \tilde{u} \|_{L^{p+\sigma}(\R^n)}^{p+\sigma} 
\aleq\ \|\nabla u\|_{L^n(B(0,3))} \|\nabla \tilde{u}\|_{L^{p+\sigma}(\R^n)}^{p+\sigma} + \brac{\|\nabla u\|_{L^{p}(B(0,3))}^{p-1} +\|\nabla u\|_{L^{p}(B(0,3))}^{p-2}} \|\nabla \tilde{u}\|_{L^{p+\sigma}(\R^n)}^{1+\sigma}.
 \end{split}
\]
With the smallness assumption \eqref{eq:smallnessassumption}, and Young's inequality for exponents $\frac{1+\sigma}{p+\sigma} + \frac{p-1}{p+\sigma} =1$ with $\delta>0$, we find for all $p \in [n,p_1]$
\[
\begin{split}
 \|\nabla \tilde{u} \|_{L^{p+\sigma}(\R^n)}^{p+\sigma} 
\leq 
&\ C(n,\mathcal{N}) \brac{\eps \|\nabla \tilde{u}\|_{L^{p+\sigma}(\R^n)}^{p+\sigma}
+ \delta \|\nabla \tilde{u}\|_{L^{p+\sigma}(\R^n)}^{p+\sigma}} \\
& + C(\delta,n,\mathcal{N}) \brac{\|\nabla u\|_{L^{p}(B(0,3))}^{p-1} +\|\nabla u\|_{L^{p}(B(0,3))}^{p-2}}^{\frac{p+\sigma}{p-1}}.
 \end{split}
\]
Now, since $\frac{p+\sigma}{p-1} \le \frac{n+1}{n-1}$ for $p\in[n,p_1]$, provided that $\delta$ and $\eps$ are small enough, we obtain
\[
\|\nabla \tilde{u} \|_{L^{p+\sigma}(\R^n)}^{p+\sigma} \aleq_{\delta,n} \brac{1+\|\nabla u\|_{L^{p}(B(0,3))}^{p-1}}^{\frac{n+1}{n-1}}.
\]
In particular, since $\eta \equiv 1$ in $B(0,1)$,
\[
\|\nabla u \|_{L^{n+\sigma}(B(0,1))}^{n+\sigma} \aleq_{\delta,n} \brac{1+\|\nabla u\|_{L^{p}(B(0,3))}^{p-1}}^{\frac{n+1}{n-1}}.
\]
\end{proof}

We combine \Cref{la:iwaniec} with the following result obtained in \cite[Theorem 4.1]{BDW20}. See \cite[Remark 4.2]{BDW20} for arbitrary dimension and systems. 
\begin{theorem}\label{th:anna}
Let $n \geq 2$. There exists $s_0 = s_0(n)> 0$ such that the following holds. Let $p \in [n,n+1]$, $f \in L_{loc}^{1}(B(0,2))$ and $u \in W^{1,p}(B(0,2),\R^N)$ be a solution to 
\[
 \div(|\nabla u|^{p-2} \nabla u) = f \quad \text{in $B(0,2) \subset \R^n$}.
\]
Then for any $s \in (0,s_0]$ and $c_0>1$ such that $\frac{np'}{n+(1-s)p'} \geq c_0$, there exists a constant $C = C(n,s,c_0)>0$ such that the following holds
\begin{equation}\label{eq:Annainequality}
 \left[ |\nabla u|^{p-2} \nabla u \right]_{W^{s,p'}(B(0,1))} \leq C\,\|f\|_{L^{\frac{np'}{n+(1-s)p'}}(B(0,2))} + C\, \|\nabla u\|_{L^{p}(B(0,2))}^{p-1}.
\end{equation}
\end{theorem}
In the above result, the dependence of the constants $s_0$ and $C$ can be chosen continuously in $p\in [n,n+1]$ by following the proof of \cite[Theorem 4.1]{BDW20}, taking into account \cite[Remark 4.2]{BDW20}. Additionally, inequality \eqref{eq:Annainequality} corresponds to inequality \cite[(4.2)]{BDW20} after applying the Sobolev embedding. 
Now \Cref{th:epsreg} follows from  
\Cref{th:anna} combined with \Cref{la:iwaniec}. 
\begin{proof}[Proof of \Cref{th:epsreg}]
Again, for simplicity of presentation, we may assume that $B(x_0,r)$ is a Euclidean ball $B(0,100)$.

Again, as in \Cref{s:prelim}, since the exponent $p_0$ from the statement of the theorem satisfies $p_0>n$, Morrey’s embedding together with \cite[Section~3]{HLp} (or \cite[Theorem 1.3]{M2025}) yields $u\in C^{1,\alpha}$ for some $\alpha>0$, and in particular $\nabla u\in L^\infty$.
Consequently, the parameter $p_0$ appearing in \Cref{la:iwaniec} can be chosen arbitrarily above $n$; for instance, taking $p_0=n+1$ in \Cref{la:iwaniec}, we obtain the existence of $\sigma\in(0,1)$ and $p_1\in(n,n+1)$ such that, if $p\in[n,p_1]$ we have
\[
\|\nabla u\|_{L^{n+\sigma}(B(0,20))} \aleq_\Lambda 1.
\]

Possibly shrinking $p_1$, we may further assume that $\frac{n+\sigma}{p}>\frac{2n+\sigma}{2n}>1$ for all $p\in[n,p_1]$.
With this choice, $u$ satisfies
\[
\div \bigl(|\nabla u|^{p-2}\nabla u\bigr)=f,
\]
where
\begin{equation}\label{eq:estimateonlittlef}
\|f\|_{L^{\frac{n+\sigma}{p}}(B(0,20))}
\aleq
\|\nabla u\|_{L^{n+\sigma}(B(0,20))}^{p}
\aleq_\Lambda 1.
\end{equation}
Let $s_0>0$ be given by \Cref{th:anna}. Shrinking $s_0$ if necessary, we may assume that
\[
0<s\le s_0
\quad\Longrightarrow\quad
s<\frac{n\sigma}{2n+\sigma}.
\]
Fix such an $s$. Then, for all $p\in[n,p_1]$,
\[
\frac{n}{n-s}\le \frac{2n+\sigma}{2n},
\qquad \text{and} \qquad 
\frac{np'}{n+(1-s)p'}\le \frac{2n+\sigma}{2n} < \frac{n+\sigma}{p}.
\]
Moreover, by possibly shrinking $p_1$ once more so that $p_1<\frac{n}{1-s}$, we ensure that
\[
\frac{np'}{n+(1-s)p'}>1
\qquad\text{for all } p\in[n,p_1].
\]
Hence, we first apply \eqref{eq:Annainequality}, then H\"{o}lder inequality, and finally \eqref{eq:estimateonlittlef}, which yields
\[
\bigl[\,|\nabla u|^{p-2}\nabla u\,\bigr]_{W^{s,p'}(B(0,15))}
\aleq_\Lambda
\|f\|_{L^{\frac{n+\sigma}{p}}(B(0,30))}+1
\aleq_\Lambda 1.
\]
Moreover, the following inequality holds for $p \geq 2$ (see for instance \cite[Lemma 2.3.7]{M})
\[
 \abs{|\nabla u|^{p-2}(x) \nabla u(x)-|\nabla u|^{p-2}(y) \nabla u(y)} \ageq_{p} \abs{\nabla u(x)-\nabla u(y)}^{p-1}.
\]
We find that
\[
 \int_{B(0,15)}\int_{B(0,15)} \frac{\abs{\nabla u(x)- \nabla u(y)}^{p}}{|x-y|^{n+\frac{s}{2(p-1)} p}}\, dx\, dy \aleq_{\Lambda}1. 
\]
Setting $t_0 = \frac{s}{2(p_1-1)} \leq \frac{s}{2(p-1)}$ we conclude.
\end{proof}

\section{Conclusion: Proof of Theorem~\ref{th:main}}\label{s:main}

\begin{proof}[Proof of \Cref{th:main}]
Let $p_0=n+1$ and let us take $\eps$ from \Cref{th:epsreg}. Recall that we assume that $k \geq 1$ (because  $k=0$ is known already). That is, we can apply \Cref{co:minmaxpharmmaps} to find a sequence $p_j \xrightarrow{j \to \infty} n^+$ and a sequence of nontrivial $p_j$-harmonic maps $u_{p_j} \in C^{1,\alpha_j}(\S^n,\n)$. There are two possible cases.

\textsc{Case 1:} \underline{There are no bubbles}, i.e., for all $x \in \S^n$ there exists $\rho_x > 0$ such that
 \[
  \limsup_{j \to \infty} \int_{B(x,\rho_x)} |\nabla u_{p_j}|^{n} < \eps^n.
 \]
In this case we cover the compact set $\S^n$ by finitely many balls $B(x,\rho_x/2)$, apply \Cref{th:epsreg} on each $B(x,\rho_x/2)$,  and from the embedding $W^{t_0,p_j}_{loc} \hookrightarrow W^{\frac{t_0}{2},n}_{loc}$ obtain\footnote{For $s\in(0,1)$ in general there is \emph{no} embedding 
$W^{s,p}_{\loc}\hookrightarrow W^{s,q}_{\loc}$ when $p>q$; see \cite{Mironescu-Sickel}.}
\[
 \sup_j [\nabla u_{p_j}]_{ W^{\frac{t_0}{2},n}(\S^n)} < \infty.
\]
Thus, up to a subsequence, 
\[
 u_{p_j} \to u \quad \text {strongly in } W^{1,n+\sigma} \text{ for some $\sigma > 0$.}
\]
By Vitali's convergence theorem we then have 
\[
 \int_{\S^n} |\nabla u|^n = \lim_{j \to \infty} \int_{\S^n} |\nabla u_{p_j}|^{p_j} >0.
\]
By the strong convergence, we can also pass to the limit in the equation \eqref{eq:pharmonic} and conclude that $u$ is an $n$-harmonic map. The map $u$ cannot be trivial since $\|\nabla u\|_{L^n} > 0$. Since $u \in W^{1,n+\gamma}$ we know that $u$ is H\"{o}lder continuous. Thus, as in \cite[Section 3]{HLp}, $u \in C^{1,\alpha}$ is a nontrivial $n$-harmonic map.

\textsc{Case 2}: \underline{There is an isolated bubble}, i.e., 
there exists $\bar{x} \in \S^n$ such that 
\begin{equation}\label{eq:eqeqe}
 \forall \varrho>0, \qquad \liminf_{j\to \infty} \int_{B(\bar{x},\varrho)} |\nabla u_{p_j}|^n \ge \eps^n
\end{equation}
and 
\begin{equation}\label{eq:eqeqebdd}
 \limsup_{j\to \infty} \int_{\S^n} |\nabla u_{p_j}|^n \leq M < \infty,
\end{equation}
and for some $R > 0$ and any $x \in \overline{B(\bar{x},R)} \setminus \{\bar{x}\}$ there exists $\rho_x > 0$ such that 
\begin{equation}\label{eq:eqeqebddv2}
 \limsup_{j\to \infty} \int_{B(x,\rho_x)} |\nabla u_{p_j}|^n \leq \eps^n  .
\end{equation}
Up to subsequence we then may conclude from \Cref{th:epsreg}
\begin{equation}\label{eq:limitingmap}
 u_{p_j} \xrightarrow{j \to \infty} u \quad \text{weakly in $W^{1,n}(\S^n)$},
\end{equation} 
and
\begin{equation}\label{eq:limitingmapstrong}
 u_{p_j} \xrightarrow{j \to \infty} u \quad \text{strongly in $W^{1,n}_{loc}\left(B(\bar{x},R) \setminus \{\bar{x}\}\right)$}.
\end{equation}

We argue as in \cite[Proof of Lemma~1]{Brezis_Coron_1985} and \cite[Section~5.2]{MMR} (see also \cite{Mou-Wang}).
Up to composition with a smooth diffeomorphism --- for instance the exponential map centered at $\bar x$ followed by a dilation --- we may assume that $u_{p_j}$ and $u$ are defined on the Euclidean unit ball $B(0,1)$, with $\bar x=0$ and $R=1$.
In these coordinates the maps are $p_j$-harmonic with respect to the pull-back metric $g$ induced by $\mathbb S^n$.
Since we are analyzing concentration at $\bar x=0$, the coefficients of $g$ converge in $C^\infty_{\mathrm{loc}}$ to the Euclidean metric $\delta_{ij}$ near $0$.
Throughout the remainder of this section we keep the metric $g$ implicit and do not record it explicitly in the $L^p$ norms.

For $\bar{R} \coloneqq \frac{1}{4} >0$, we define 
\[
 Q_j(\rho) \coloneqq \sup_{x\in B(0,\bar{R})} \int_{B (x,\rho)} |\nabla u_{p_j}|^{n}.
\]
We note that $Q_j$ is continuous, $Q_j(0)=0$, non-decreasing in $\rho$, and by \eqref{eq:eqeqe} for sufficiently large $j$
\[
\begin{split}
 Q_j(2\bar{R}) &= \sup_{x\in B  (0,\bar{R})}\int_{B (x,2\bar{R})} |\nabla u_{p_j}|^{n} \ge \int_{B (0,\bar{R})}|\nabla u_{p_j}|^{n}\ge \eps^{n}.
\end{split}
 \]

Thus, there exists  $\rho_j>0$ and $x_j \in \overline{B (0,\bar{R})}$ such that
\[
 Q_j(\rho_j) = \int_{B (x_j,\rho_j)} |\nabla u_{p_j}|^n = \frac{\eps^n}{2}.
\]

\underline{Claim:} $x_j \to 0$ and $\rho_j \to 0$.

Indeed, up to a subsequence, we may assume that $x_j \to \bar{y}\in \overline{B (\bar{x},\bar{R})}$ and $\rho_j \to \bar{\rho} \ge 0$. If $\bar{\rho}\neq 0$, then for sufficiently large $j$, it holds
\[
 \frac{\eps^n}{2} = Q_j(\rho_j) \ge \int_{B (0,\rho_j)}|\nabla u_{p_j}|^{n} \ge \int_{B (0,\frac{\bar{\rho}}{2})} |\nabla u_{p_j}|^{n} \ge \eps^n.
\]
This is a contradiction. Hence, we have $\rho_j\to 0 $.
Now, if $\bar{y}\neq \bar{x}$, then let $u\in W^{1,n}$ be the limiting map from \eqref{eq:limitingmap}. We choose $\rho\in (0,\rho_y)$ (where $\rho_y$ is defined in \eqref{eq:eqeqebddv2}) small enough so that
\[
 \int_{B (\bar{y},\rho)} |\nabla u|^n < \frac{\eps^n}{4} \quad \text{and}\quad \bar{x} \notin B (\bar{y},2\rho).
\]
We choose $j$ large enough so that $2(|\bar{y}-x_j|+\rho_j)\le \rho$. Then, it holds
\[
 \begin{split}
  \frac{\eps^{n}}{2} & = \int_{B (x_j,\rho_j)} |\nabla u_{p_j}|^{n} \le \int_{B \big(\bar{y}, 2(|\bar{y}-x_j|+\rho_j) \big)} |\nabla u_{p_j}|^{n} \le \int_{B (\bar{y},\rho)} |\nabla u_{p_j}|^{n}.
  \end{split}
\]
By \eqref{eq:limitingmapstrong} we have strong convergence outside of the singular set, and in particular in $B_{2\rho}(\bar{y})$. Hence, it holds
\[
\begin{split}
	\frac{\eps^{n}}{2}& \le \lim_{j\to +\infty} \int_{B (\bar{y},\rho)} |\nabla u_{p_j}|^{n} = \int_{B (\bar{y},\rho)} |\nabla u|^n < \frac{\eps^n}{4}.
\end{split}
\]
This is a contradiction, hence $x_j \to \bar{x}$, and the Claim is etablished.

Define now
\[
 \Omega_j \coloneqq \frac{  B (0,\bar{R}) -x_j }{\rho_j} \subset \R^n , \qquad v_j(x) = u_{p_j}\left( x_j + \rho_j\, x \right) \text{ for } x\in \Omega_j.
\]
Then we have $\Omega_j \to \R^n$ in the sense that for any $R>0$ there exists $j$ large enough so that $B (0,R) \subset \Omega_j$. Also $v_j\colon (\Omega_j ,g(x_j + \rho_j\cdot))\to \n$ is a $p_j$-harmonic map (because it satisfies the equation \eqref{eq:pharmonic}, which is scaling invariant).

For $j$ large enough $B (0,1)\subset \Omega_j$ and we have
\begin{equation}\label{eq:nontrivialenergy}
 \int_{B (0,1)} |\nabla v_j|^{n} = \int_{B (x_j,\rho_j)} |\nabla u_{p_j}|^{n} =  Q_j(\rho_j) = \frac{\eps^{n}}{2}.
\end{equation}
For any $x\in\R^n$ we have for $j$ large enough $B (x,1)\subset \Omega_j$ and since $\rho_j\to 0$ and $x_j \to 0$ there exists $y\in B (0,\bar{R})$ such that for sufficiently large $j$ we have $B (x_j+\rho_j x,\rho_j) \subset B (y,\rho_j)$. Hence, it holds
\[
 \int_{B (x,1)} |\nabla v_j|^{n} = \int_{B (x_j+\rho_j x,\rho_j)} |\nabla u_{p_j}|^{n} \le  \sup_{y\in B (0,\bar{R})}\int_{B (0,\bar{\rho}_j)} |\nabla u_{p_j}|^{n} = Q_j(\rho_j) =  \frac{\eps^{n}}{2}.
\]
Thus, for any $x\in \R^n$ and sufficiently large $j$ we have
\[
 \int_{B (x,1)} |\nabla v_j|^{n} \le \frac{\eps^{n}}{2}.
\]
Since this holds for any $x \in \R^n$, up to passing to a not relabelled subsequence, we can assume
\[
 v_j \xrightarrow{j \to \infty} \omega \quad \text{weakly in $W^{1,n}_{loc}(\R^n)$}.
\]
Moreover, from \Cref{th:epsreg}, combined with Sobolev embedding, we obtain that for every $x\in \R^n$
\[
 v_j \to \omega \text{ strongly in } W^{1,n+\sigma}\left(B (x,1/2 )\right).
\]
Since $v_j$ are $p_j$ harmonic maps, with the strong $W^{1, n+\sigma}-$ convergence we may pass to the limit in the equation of $v_j$ and obtain that the limiting map $\omega \in W^{1,n}_{loc}(\R^n,\n)$ also satisfies an $n$-harmonic map equation in $\R^n$.

Moreover, $\omega$ is non-trivial. Indeed, by the change of the variables (and the scaling invariance of the $n$-energy)
\[
\int_{B (0,1)}|\nabla \omega|^n = \lim_{j\to \infty} \int_{B (0,1)} |\nabla v_j|^n = \lim_{j\to \infty}\int_{B (x_j,\rho_j)} |\nabla u_{p_j}|^n = \lim_{j\to \infty} Q_j(\rho_j) = 
\frac{\eps^n}{2} \neq 0.
\]
In addition, we find that $\omega$ has finite energy. Indeed, for any $R>0$ we have
\[
 \int_{B (0,R)} |\nabla \omega|^n = \lim_{j\to \infty} \int_{B (0,R)} |\nabla v_j|^n = \lim_{j\to\infty} \int_{B (x_j,\rho_j R)} |\nabla u_{p_j}|^n \le C(n,M),
\]
where the last inequality is the uniform bound on the $n$-energy of $u_{p_j}$ from \eqref{eq:eqeqebdd}. Thus, it holds $\nabla \omega\in L^n(\R^n)$.
We also observe that as $W^{1,n+\sigma}$-limit $\omega$ is H\"{o}lder continuous, and since it solves the $n$-harmonic map 
equation it is actually $C^{1,\alpha}(\R^n,\n)$ for some $\alpha > 0$.

In order to find a nontrivial harmonic map $\tilde{\omega}\colon \S^n \to \n$ we argue as follows. Let $\tau_N\colon \S^n\setminus\{N\}\to \R^n$ be the stereographic projection from the north pole.
By conformal invariance of the $n$-harmonic map equation in the domain, the map $\tilde\omega \coloneqq \omega\circ \tau_N \colon \S^n\setminus\{N\}\to \n$
is $n$-harmonic on $\S^n\setminus\{N\}$ and satisfies
$\tilde\omega\in W^{1,n}(\S^n\setminus\{N\},\n)\cap C^{1,\alpha}(\S^n\setminus\{N\},\n)$.

To remove the singularity at $N$, consider the stereographic projection $\tau_S\colon \S^n\setminus\{S\}\to \R^n$ from the south pole and define $\omega_2 \coloneqq \tilde\omega\circ \tau_S^{-1}\colon \R^n\setminus\{0\}\to \n$.
Then $\omega_2$ is $n$-harmonic on $\R^n\setminus\{0\}$ and $
\omega_2=\omega\circ (\tau_N\circ \tau_S^{-1})=\omega\circ \kappa$, where $\kappa(x)=\frac{x}{|x|^2}$.
Since $\kappa$ is conformal and $n$ is the critical exponent, the $n$-energy is invariant under $\kappa$, and we have $\|\nabla \omega_2\|_{L^n(B(0,1))} = \|\nabla \omega\|_{L^n(\R^n\setminus B(0,1))}$.
\[
\int_{B(0,1)} |\nabla \omega_2|^n
= \int_{\R^n\setminus B(0,1)} |\nabla \omega|^n <\infty,
\]
so $\omega_2$ has finite $n$-energy in a neighborhood of the isolated singularity at $0$.
By the removable singularity theorem \cite[Theorem~5.1]{Mou-Yang-1996} (see also \cite{Duzaar-Fuchs}), $\omega_2$ extends to a map in
$C^{1,\alpha}(\R^n,\n)$ (hence also in $W^{1,n}_{\loc}(\R^n,\n)$).
Reverting the stereographic projection, we conclude that $\tilde\omega$ extends to an
$n$-harmonic map $\tilde\omega\in W^{1,n}(\S^n,\n)\cap C^{1,\alpha}(\S^n,\n)$.
Moreover, $\tilde\omega$ is non-constant since $\omega$ is non-constant.
\end{proof}

\bibliographystyle{abbrv}%
\bibliography{bib}%
\end{document}